\documentclass[10pt]{amsart}
\usepackage{amssymb}

\newtheorem{prop}{Proposition}

\let\s\mathsf
\let\Cal\mathcal
\def\Frac#1#2{{\textstyle #1\over\textstyle #2}}%
\title{Characterization of hyperbolic groups via random walks}
\author{Victor Gerasimov}
\address{Victor Guerassimov (Gerasimov), Departamento de Matem\'atica,
Universidade Federal de Minas Gerais,
Av. Ant\^onio Carlos, 6627/
CEP: 31270-901. Caixa Postal: 702
Belo Horizonte, MG,
Brazil}
\email{victor.gerasimov@gmail.com}
\author{Leonid Potyagailo}
\address{ Leonid Potyagailo, UFR de Math\'ematiques, Universit\'e
de Lille, 59655 Villeneuve d'Ascq cedex, France}
\email{leonid.potyagailo@univ-lille.fr}

\thanks{This work was
supported by CEMPI LABEX project (UMR 8524) of the
University of Lille; and the Max Planck Institute for Mathematics, Bonn}

\begin{document}
\begin{abstract}
Our first result gives a partial converse to a well-known theorem of A.
Ancona for hyperbolic groups.
We prove that a group $G$, equipped with a symmetric probability measure
whose finite support generates $G$, is hyperbolic if it is nonamenable
and satisfies the following condition:\hfil\penalty-10000
for a sufficiently small $\varepsilon >0$
 and $r\geqslant0$, and
for every triple $(x, y, z)$, belonging to a word
geodesic of the Cayley graph, the probability that a random path from
$x$ to $z$ intersects the closed ball of radius $r$ centered at $y$ is
at least $1-\varepsilon.$

We note that if a group is hyperbolic then the above
 condition for $r=0$ is satisfied
by Ancona's theorem and for any $r>0$ follows from this paper.

Another our theorem claims that a finitely generated group is hyperbolic
if and only if the probability that a random path, connecting two
antipodal points of an open ball of radius $r$ does not intersect it
is exponentially small with respect to $r$ for $r\gg0$..
The proof is
based on a  purely geometric criterion for the hyperbolicity of a connected graph.
\end{abstract}

%
%
%
\date{\today}
\maketitle
\markboth{V.~Gerasimov, L.~Potyagailo, {\sl Ancona's property and hyperbolicity}\quad(\today)}{V.~Gerasimov, L.~Potyagailo,
{\sl Ancona property and hyperbolicity}\quad(\today)}

\def\mpPath{}
\def\-{\hbox{-}}
\section{Introduction}
\subsection{History and description of the results}
The techniques of  random walks
in describing geometrical properties
of manifolds and their fundamental groups
appeared in the classical works of H.~Furstenberg. His approach to study group boundaries using random walks made a breakthrough in the
rigidity theory of discrete groups (see [Fu71] and references therein)

After Gromov's introduction of word hyperbolic groups \cite{Gr87},
A.~Ancona, continuing Furstenberg's approach,
discovered an important probabilistic property of random walks on hyperbolic groups
\cite{Anc87}, \cite{Anc88}.
One of his famous results claims that if a group $G$ is hyperbolic then,
for every triple $(x,y,z)$ of elements of $G$, situated in this order on a word geodesic,
their Gromov product with the respect to the Green metric is uniformly bounded.

The notion of the Green metric appeared in the theory of harmonic functions on Riemannian manifolds.
In the context of random walks the Green distance between $x$ and $y$ is
a function of the probability that the random walk started at $x$ will reach $y$ (see Subsection 3.1).
Let $G$ be a finitely generated group equipped with a symmetric probability
 measure $\mu$ with a finite support which generates $G$. We call such a measure  \textit{admissible},
so every group possessing such a measure is automatically finitely generated.

The result of Ancona is equivalent to that for a group $G$ equipped with an admissible measure
 $\mu$ there exists $\varepsilon$ in the open interval $(0,1)$ such that  the following property
 (which we call \textit{weak Ancona property}, see Section 3) holds:
\vskip3pt
$\s{WA}_\varepsilon$: \textsl{for every elements $x,z$ of a group $G$
and every $y$ that belongs to the word geodesic segment $[x,z]$,
the probability that a random path between $x$ and $z$ passes through $y$ is at least $1-\varepsilon$.}
\vskip3pt
Using this property, Ancona proved
that if $G$ is hyperbolic  then
the Martin boundary of $G$ corresponding to a random walk with respect to an admissible measure $\mu$
is naturally homeomorphic to the Gromov boundary of $G$.

This result possesses generalizations to a larger classes of groups  \cite{GGPY21}.
\vskip3pt
The following conjecture is the main motivation for this paper:\vskip3pt
\textbf{Conjecture.} \textsl{Let $G$ be a group possessing an admissible measure satisfying $\s{WA}_\varepsilon$
for some $\varepsilon\in(0,1)$. Then $G$ is hyperbolic.}
\vskip3pt
We give a partial positive answer to this conjecture.
We consider the following stronger property:
\vskip2pt
$\s{TA}_{\varepsilon,r}$ (thick Ancona property):
 \textsl{for every $x,z\in G$
and every $y\in[x,z]$
the probability that a random path between $x$ and $z$ intersects the closed
ball $B_{\leqslant r}y$ of radius $r$ centered at $y$, is at least $1-\varepsilon$.}
\vskip3pt
The property $\s{WA}_\varepsilon$ coincides with $\s{TA}_{\varepsilon,0}$.
If $\s{TA}_{\varepsilon,r}$ holds then for every $r'\leqslant r$
there exists $\varepsilon'\in(0,1)$ (depending on $\varepsilon$) such that $\s{TA}_{\varepsilon',r'}$ holds
and in particular $\s{WA}_{\varepsilon'}$ holds
(see Subsection 3.2).

The main result of the paper is the following:

\vskip3pt
\textbf{Theorem 3}. \textsl{Let $G$ be a non-amenable group
equipped with an admissible measure such that
for a sufficiently small positive $\varepsilon$ and some $r\geqslant0$
the property $\s{TA}_{\varepsilon,r}$ holds. Then $G$ is hyperbolic}.
\vskip3pt

The proof of Theorem 3 is constructive.
We introduce in it two  conditions $\s A$ and $\s B$, containing numerical parameters which depend  only on $G$ and $\mu$.
Using these parameters, we obtain a critical value $\varepsilon_0$ depending on $G$ and $\mu$, such that every
$\varepsilon\in(0,\varepsilon_0)$ is appropriate.
Furthermore, we prove in Proposition 6 that both 
conditions   $\s A$ and $\s B$
  follow from the non-amenability assumption.
So Proposition 6 gives a stronger statement than Theorem 3, where the non-amenability assumption  is replaced by weaker conditions $\s A$ and $\s B$.

The  property $\s{TA}$ or $\s{WA}$ valid for some $\varepsilon$,
may not hold for a smaller value of $\varepsilon$.
 However if $G$ is hyperbolic and the measure $\mu$ is admissible then for every $\varepsilon\in(0,1)$ there exists
$r\geqslant0$ such that $\s{TA}_{\varepsilon,r}$ holds (Proposition 1 in Section 3). So Theorem 3 implies the following:
\vskip3pt
\textbf{Corollary}\sl. Let $G$ be a non-amenable  group equipped with an admissible measure $\mu$. Assume that  the property
 $\s{TA}_{\varepsilon,r}$ holds  for sufficiently small $\varepsilon$ (less than the above $\varepsilon_0$) and some $r\geqslant0$.
 Then for every $\varepsilon\in(0,1)$ there exists $r=r(\varepsilon) \geqslant0$ such that $\s{TA}_{\varepsilon,r}$ holds. \rm
\vskip3pt
It is worth mentioning that the Ancona property $\s{WA}$ was used by P.~Ha\"issinsky and P.~Matthieu
to obtain  a much shorter proof of the well-known Baum-Connes conjecture   for hyperbolic groups \cite{HM14}.
The original proofs were given independently by V.~Lafforgue \cite{La02} and I.~Mineyev and G.~Yu \cite{MY02}.
Using the hyperbolicity of $G$ and Ancona's property the authors of \cite{HM14}
proved the conjecture by considering the action of $G$ on the Martin compactification.
Note that the Baum-Connes conjecture is true for amenable groups.
Thus Theorem 3 implies that if the group is either amenable
or non-amenable having the property $\s{TA}_{\varepsilon,r}$ for
$\varepsilon\in(0,\varepsilon_0)$, where $\varepsilon_0$ is as above, then the Baum-Connes conjecture is true for $G$.
\vskip3pt
We provide another characterization
of hyperbolic groups (and graphs) in terms of random walks.
It does not require the nonamenability assumption.

Let $(G,\mu)$ be an admissible pair.
Denote by $\varepsilon_0$ the minimal value of $\mu$ on its finite support $\mathcal S$.

For $x,z\in G$ let $y$ be the midpoint
of the geodesic interval $[x,z]$ of the Cayley graph $\s{Ca}(G,\Cal S)$.
Thus $\{x,z\}$ is a diametral pair of the ball $B_{\leqslant r}(y)$ where $r$
denotes the distance from $y$ to $x$ and $z$.
We have
\vskip3pt
\textbf{Theorem 2}
 (Proposition 4, Section 5). \textsl{The group $G$ is hyperbolic
 if and only if, for some $\varepsilon\in(0,\varepsilon_0)$
and sufficiently big $r$,
the probability that a random path from $x$ to $z$
not intersecting the open ball $B_{<r}(y)$
is at most $\varepsilon^{l0r}$.}
\vskip3pt
The proof of Theorem 2 uses the following geometric hyperbolicity
 criterion for a graph $\Gamma$.

For a geodesic interval $I=[x,z]\subset\Gamma$
of length $2r$
with the midpoint $y$, let\newline
 $\pi_I\leftrightharpoons\s{inf}\{\ell_\gamma:\gamma $ is a rectifyable curve
joining $x$ and $z$, that does not intersect the ball $B_{<r}y\}$.

Denote $\pi(\Gamma)=\underset{r\rightarrow\infty}{\s{lim\,inf}}\pi_I$.
\vskip3pt
\textbf{Theorem 1.} (Proposition 3, Section 4) \textsl{If
$\Gamma$ is connected and $\pi(\Gamma)>10$ then $\Gamma$ is hyperbolic.}
\vskip3pt
The converse to Theorem 1 is also true.
Moreover, if $\Gamma$ is hyperbolic then $\pi(\Gamma)=\infty$ \cite[7.1A]{Gr87}.
\subsection{The structure of the article}
In Section 2 we give the definitions and introduce the notation needed for the exact formulations
of our results. In Section 3 we discuss versions of Ancona's property.
In Section 4 we discuss the spectral radius of the Green series and the non-amenability condition.
In Section 5 we prove Theorems 1 and 2.
In Section 6 we prove a technical lemma about 2-Lipschitz functions.
The proof of Proposition 6, implying Theorem 3 is given in Section 7.
\subsection{Acknowledgements}
The authors thank the scientific program CEMPI LABEX (UMR 8524) of the University of Lille for several research invitations and financial support of both of us during our work on this paper.

This article was completed during our three months stay at the Max Planck Institute for Mathematics in Bonn.
We are grateful to the Max Planck Society for awarding us the Research Grants and for
the excellent conditions for our work.
\section{Conventions and notation}
\subsection{General conventions}
The symbol `$\leftrightharpoons$' means `is equal by definition'.

For a function $f$ defined on a set $S$ we say that $S$ is the \it domain \rm$\s{dom}(f)=S$.\hfil\penalty-10000
For a set $T\subset\s{dom}(f)$ we denote by $f|_T$ the restricion of $f$ over $T$.
So, $\s{dom}(f|_T)=T$.

For the value $f(x)$ ($x\in\s{dom}(f)$), we somtimes use the abbreviations $fx$ and $f_x$.

For $a,b\in\Bbb R$ we denote
the \it closed interval \rm$[a,b]\leftrightharpoons\{x\in\Bbb R:a\leqslant x\leqslant b\}$
and the \it open interval \rm$]a,b[\leftrightharpoons\{x\in\Bbb R:a<x<b\}$ (we prefere to reserve
the notation `$(a,b)$' for ordered pair).

Considering a fixed metric space we denote by $B_{\leqslant r}p$ and $B_{<r}p$
respectively the closed and the open ball of radius $r$ centered at $p$.

We numerate the formulas independently in each section.

\subsection{Geodesics}
Considering a fixed metric space $M$ we denote by $|p\-q|$
the distance between points $p,q\in M$.\par
A \it geodesic path \rm or simply a \it geodesic \rm
is a map of the form $[0,L]\overset{\gamma}\to M$, where $L\in\Bbb R_{\geqslant0}$,
such that $|\gamma(0)\-\gamma(x)|=x$ for all $x\in[0,L]$. The number $L$ is called the \it length \rm
of the geodesic $\gamma$ and is denoted by $\ell\gamma$ or by $\ell_\gamma$.
The ordered pair $\partial\gamma{\leftrightharpoons}(\gamma(0),\gamma(\ell_\gamma))$ it the \it boundary \rm of $\gamma$.
The points $\gamma(0)$ and $\gamma(\ell_\gamma)$ are the \it endpoints \rm of $\gamma$.

A space $M$ is \it geodesic \rm if, for every $(p,q)\in M^2=M\times M$, there exists a geodesic $\gamma$
with $\partial\gamma=(p,q)$.
\subsection{Geodesic bigons}
In a metric space $M$, a (geodesic) \it bigon \rm is an ordered pair $\beta{=}(\beta_0,\beta_1)$
of geodesics with the same boundary. The geodesics $\beta_\iota$ are the \it sides \rm$\beta$.

 Denote by $\s{gb}(M)$ the set of geodesic bigons in $M$.

For $\beta\in\s{gb}(M)$, the \textit{length} of $\beta$ is the value $\ell(\beta_0)=\ell(\beta_1)$.\newline
The function $\s w_\beta:[0,\ell_\beta]\ni s\mapsto|\beta_0(s)\-\beta_1(s)|$ is the \it width function \rm
of $\beta$.

For a constant $C\geqslant0$, we say that \it the bigons in $M$ are $C$-thin \rm
if $\s w_\beta\leqslant C$ for every $\beta\in\s{gb}M$.

It is easy to prove that if $M$ is $\delta$-hyperbolic then the bigons in $M$ are $C$-thin
for some constant $C$ that depends only on $\delta$.
The converse is true if $M$ is a metric graph (see below).
\vskip3pt
As in \cite{GP24} we will need the following simple lemma.
\vskip3pt
\sc Lemma\sl.
If $\beta=(\beta_0,\beta_1)\in\s{gb}M$, $s\in[0,\ell_\beta]$ and $\s w_\beta(s)=W$,
then the distance $d$ from $\beta_\iota(s)$ to $\s{Im}\beta_{1-\iota}$
is at least $W/2$\rm.
\vskip3pt
\textit{Proof.}
Let $t$ be such that $|\beta_\iota(s)\-\beta_{1-\iota}(t)|=d$.
For the triangle with vertices
$\beta_0(0){=}\beta_(0),\beta_\iota(s),\beta_{1-\iota}(t)$,
the triangle inequality gives:
$d\geqslant|s-t|$.

For the triangle with vertices
$\beta_0(s),\beta_1(s),\beta_1(t)$
the triangle inequality gives:\newline
 $W=|\beta_\iota(s)\-\beta_{1-\iota}(s)|\leqslant
|\beta_\iota(s)\-\beta_{1-\iota}(t)|+|\beta_{1-\iota}(t)\-\beta_{1-\iota}(s)|=d+|s-t|$.

The two inequalities imply $W\leqslant2d$.\hfill$\square$
\subsection{Metric graphs and normalization} A \it metric graph \rm is a one-dimentional connected $\s{CW}$-complex
equipped with a path-metric. By defaults the length of every edge of a graph is equal to one.

Our main tool to prove the hyperbolicity is the following theorem.
\vskip3pt
\textbf{Theorem P.} (P. Papasoglu \cite{Pap95}) \sl
If, in a metric graph $\Gamma$, the bigons are $C$-thin then $\Gamma$ is $\delta$-hyperbolic
for some constant $\delta$ that depends only on $C$\rm.
\vskip3pt
Since we are going to study random walks as stochastic processes with discrete time and the state space
being the set of vertices of a graph, we need to adapt Papasoglu's ``continuous'' approach
to the discrete case. Namely we need to restrict ourselves to the bigons $(\beta_0,\beta_1)$ with
$\beta_0(0){=}\beta_1(0)$ being a vertex of a graph. (We do not need that the final point
$\beta_0(\ell_\beta)=\beta_1(\ell_\beta)$ is a vertex.)

However we allow loops and muliple edges.

Let us call a geodesic bigon $(\beta_0,\beta_1)$ in a metric graph $\Gamma$ \textit{normalized}
if its initial point $\beta_0(0)=\beta_1(0)$ is a vertex of $\Gamma$.
\vskip3pt

Let us call a geodesic bigon $(\beta_0,\beta_1)$ \textit{regular} if the equation $\beta_0(t)=\beta_1(t)$ is
satisfied only at the initial point and the final point.

If all regular bigons in a graph are $C$-thin then all bigons are $C$-thin.
 So one can consider only regular bigons in
the statement of Theorem P.

The union of the sides of a regular bigon is a topological circle in a graph. Thus the length of such circle
is integer and hence the length of a regular bigon is a multiple of $1\over2$.

The following simple lemma allows one to require in Papasoglu's theorem that the bigons are regular.
\vskip3pt
\textsc{Lemma.} \textsl{If, in a metric graph $\Gamma$ all regular normalized bigons are $C$-thin then all geodesic
bigons in $S$ are $C{+}1$-thin.}
\vskip3pt
\textit{Proof.}
Consider a regular bigon $\beta=(\beta_0,\beta_1)$ with endpoints $p,q$ neither of which is
 a vertex of $\Gamma$.
Let us declare the initial point of the bigon one of $p,q$
which is not further away from the vertices of $\Gamma$
than the other.
That is, if $|p\hbox{-}q|={m\over2}$ and $m$ is odd then this minimal distance (denote it by $\delta$) ia at most $1\over4$,
and if $m$ is even then $\delta$ is at most $1\over2$.

Without loss of generality we can assume
that $|p\hbox{-}p'|=\delta$\begin{picture}(0,0)(47,-17)
\put(20,-100){
\pdfximage width 200pt{\mpPath 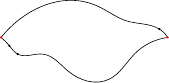}\pdfrefximage\pdflastximage}
\put(211,-32){$p$}
\put(225,-49){$p'$}
\put(45,-63){$q$}
\put(35,-52){$q'$}
\put(133,-110){\makebox(0,0)[cc]{Figure 1}}
\end{picture}\hfil\penalty-10000
 and $p'$ is a vertex of $\Gamma$
 belonging to the side $\beta_0$. Let\hfil\penalty-10000
 $q'$ be the unique point in
 the closure of  the edge\hfil\penalty-10000
 containing $q$ such that
 $|p'\hbox{-}q'|=|p\hbox{-}q|$
(see Figure 1).\hfil\penalty-10000

Since the bigon $\beta$ is regular, the edge  containing $q$\hfil\penalty-10000
 is not a loop.
Let $\beta'=(\beta_0',\beta_1')$ be
 the bigon obtaned
 from $\beta$\hfil\penalty-10000
 by the ``rotation'' of $\beta$ in the
 direction
from $p$ to $p'$
 and hence\hfil\penalty-10000
 from $q$ to $q'$.
The bigon $\beta'$
 is geodesic, regular,
 and normalized.\hfil\penalty-10000
If the width of $\beta'$ is at most $C$
 then, since $\delta\leqslant{1\over2}$, by triangle\hfil\penalty-10000
inequality, the width of $\beta$ is at most $C+1$.\hfill$\square$
\vskip3pt
We also need to normalize the balls.
A ball in $\Gamma$ is called \textit{normalized} if its center is a vertex of $\Gamma$.

It follows from the triangle inequality that every ball of radius $r$ contains a normalized ball of radius $r-1$.
Thus if a graph satisfies $\s{TA}_{\varepsilon,r}$ for the normalized balls then it satisfies
$\s{TA}_{\varepsilon,r+1}$ for all balls.
\subsection{Admissible probability measures on groups}
Let $G$ be an arbitrary (discrete) group and let $\mu$ be a probability measure on $G$.\par
\textbf{Definition.} A measure $\mu$ is \textit{admissible} if the following conditions hold:\hfil\penalty-10000
--- $\mu$ is \it symmetric\rm: $\mu(g)=\mu(g^{-1})$ for every $g\in G$;\hfil\penalty-10000
--- the \it support \rm$\s{supp}\mu\leftrightharpoons\mu^{-1}\Bbb R_{>0}$ is finite:\hfil\penalty-10000
--- $\s{supp}\mu$ is a generating set for $G$.

If a group possesses an admissible measure then it is finitely generated.

A pair $(G,\mu)$ is \it admissible \rm if $G$ is a group and $\mu$ is an admissible probability measure on $G$.
\subsection{Words and paths in the Cayley graph}
In the subsections 2.6--2.9 we consider a fixed admissible pair $(G,\mu)$.
Denote
$\Cal S\leftrightharpoons\s{supp}\mu$.

Let $\langle\Cal S\rangle$ denote the monoid freely generated by $\Cal S$.
The elements of $\langle\Cal S\rangle$ are the words in the alphabet $\Cal S$, including
 the empty word that we denote by $\lozenge$.
For $\psi\in\langle\Cal S\rangle$, denote $\ell_\psi\leftrightharpoons$the length of $\psi$,
that is the number of letters in $\psi$.
So, $\ell_\psi=0\Leftrightarrow\psi=\lozenge$.

Since $\mu$ is symmetric the inversion $s\mapsto s^{-1}$ on $\Cal S$
extends over the operation $\psi\mapsto\psi^{-1}$ on $\langle\Cal S\rangle$
with $(\psi^{-1})^{-1}=\psi$ and $\psi{\cdot}\psi^{-1}=\lozenge\Leftrightarrow\psi=\lozenge$.

We define a left invariant partial order on $\langle\Cal S\rangle$:
 $\psi\leqslant_{\fam0l}\psi_1$ if $\psi$ is an initial segment of $\psi_1$.\newline
 So $\lozenge\leqslant_{\fam0l}\psi\leqslant_{\fam0l}\psi$.

We also need the right invariant partial order $\psi\leqslant_{\fam0r}\psi_1$ if $\psi^{-1}\leqslant_{\fam0l}\psi_1^{-1}$.
We need the following simple observation.
\vskip3pt
\sc Lemma\sl.
For $\sigma,\sigma_1,\theta,\theta_1\in\langle\Cal S\rangle$,
if $\sigma,\sigma_1$ are incomparable with respect to the order
`$\leqslant_{\fam0 l}$'
or $\theta,\theta_1$  are incomparable with respect to the order
`$\leqslant_{\fam0 r}$',
then $\sigma\langle\Cal S\rangle\theta\cap\sigma_1\langle\Cal S\rangle\theta_1=\varnothing$\rm.
\vskip3pt
\textit{Proof.}
For $\psi\in\sigma\langle\Cal S\rangle\theta\cap\sigma_1\langle\Cal S\rangle\theta_1$
there exist $\varphi,\varphi_1\in\langle\Cal S\rangle$
such that $\psi=\sigma\varphi\theta=\sigma_1\varphi_1\theta_1$,
so $\sigma$ and $\sigma_1$ are comparable with respect to `$\leqslant_{\fam0 l}$'
and $\theta$ and $\theta_1$ are comparable with respect to `$\leqslant_{\fam0 r}$'.\hfill$\square$
\vskip3pt
For $\psi\in\langle\Cal S\rangle$ define $\psi^\downarrow\leftrightharpoons\{\sigma\in\langle\Cal S\rangle:\sigma\leqslant_l\psi,\sigma\ne\psi\}$.
In particular, $\psi^\downarrow=\varnothing\Leftrightarrow\psi=\lozenge$.

For a word $\psi=\prod_{i=1}^{\ell\psi}s_i\in\langle\Cal S\rangle$ ($s_i{\in}\Cal S$)
 denote by $\boldsymbol\pi(\psi)=\boldsymbol\pi_\psi$ the value of the word $\psi$ in $G$.
In particular $\boldsymbol\pi_{\lozenge}=\boldsymbol\iota\leftrightharpoons\boldsymbol\iota_G$
the neutral element of $G$.

So, we have the homomorphism
$\langle\Cal S\rangle\overset{\boldsymbol\pi}\to G$.
The \it word norm \rm$|g|$ (with respect to $\Cal S$) is the number $\s{min}\{\ell_\psi:\psi{\in}\pi^{-1}g\}$.

The word norm induces the left invariant \it word distance \rm $|x\-y|=|x^{-1}y|$.

We interprete the elements of the set $G{\times}\langle\Cal S\rangle$
as paths in the Caley graph $\s{Ca}(G,\Cal S)$: $(x,\psi)\in G{\times}\langle\Cal S\rangle$
is a path from the vertex $x$ to the vertex $x{\cdot}\boldsymbol\pi_\psi$.

Let $\s{Pth}(x,z)\leftrightharpoons\{x\}{\times}\boldsymbol\pi^{-1}(x^{-1}z)$.
We will often identify the paths in $\s{Pth}(x,z)$ and the words in $\boldsymbol\pi^{-1}(x^{-1}z)$
without confusing.
\subsection{Weight of paths and convolution powers}
For a word $\psi=s_1\dots s_n$ ($s_i\in\Cal S$) define its $\mu$\it-weight \rm as $\psi_\mu{\leftrightharpoons}\prod_{i=1}^n\mu(s_i)$.
We postulate that the $\mu$-weight of the empty word $\lozenge$ is 1.

For $n\in\Bbb N$ define a probability measure on $G$ called the $n$\it-th convolution power \rm of $\mu$ as follows:\newline
$\mu^{(n)}(g)\leftrightharpoons\sum\{\mu_\psi:\psi\in\pi^{-1}g,\ell_\psi=n\}$ ($g\in G$).

In particular, $\mu^{(1)}{=}\mu$ and
$\mu^{(0)}$ is the Dirac probability measure concentrated at the single element $\boldsymbol\iota_G$.

In the context of random walk the number $\mu^{(n)}(g)$ is interpreted as the probability
of reaching the vertex $g$ in exactly $n$ steps of the random walk started at $\boldsymbol\iota_G$.

In a similar way, as for the word norm and the word metric, it is convenient
to consider the two-variable version of the convolution powers:\hfil\penalty-10000
 $\mu^{(n)}(x,y)\leftrightharpoons\mu^{(n)}(x^{-1}y)$
and interprete them as the probability of reaching the vertex $y$
of the Cayley graph in exactly $n$ steps of the random walk started at the
vertex $x$.
\subsection{Transience}
A group is called \it transient \rm with respect to a measure $\mu$ if the sum\hfil\penalty-10000
$\s{Gr}_\mu(g)\leftrightharpoons\sum_n\mu^{(n)}(g)=\sum\{\psi_\mu:\psi\in\pi^{-1}g\}$ is finite
for some, and therefore for any element $g\in G$.

The two-variable version: $\s{Gr}_\mu(x,y)\leftrightharpoons\s{Gr}_\mu(x^{-1}y)$.

It turns out that the property of transience does not depend on the choice of admissible measure $\mu$.
Namely, by N. Varopoulos theorem \cite{Va86}, a finitely generated group is not  transient if and only if
it is virtually abelian of rank at most two.
\subsection{Green probability space of paths with fixed endpoints}
Assuming that the group $G$ is transient, on each set $\pi^{-1}g$ ($g\in G$),
define the probability measure $\Bbb P_{\boldsymbol\iota,g}$
such that $\Bbb P_{\boldsymbol\iota,g}\{\psi\}=\Frac{\psi_\mu}{\s{Gr}_\mu(g)}$.

Similarly, we have a two-variable version $\Bbb P_{x,z}$ on $\s{Pth}(x,z)$.

We need the following expression for conditional probability.
For $\sigma,\theta\in\langle\Cal S\rangle$ and $g\in G$, denote
 $h_{g,\sigma,\theta}\leftrightharpoons\boldsymbol\pi(\sigma^{-1})g\boldsymbol\pi(\theta^{-1})$.\hfil\penalty-10000
$\Cal H_{g,\sigma,\theta}\leftrightharpoons
(\sigma{\cdot}\langle\Cal S\rangle{\cdot}\theta)\cap\boldsymbol\pi^{-1}g=
\sigma{\cdot}\boldsymbol\pi^{-1}(h_{g,\sigma,\theta}){\cdot}\theta$.
\vskip3pt
\sc Lemma\sl.
For every $g{\in}G$, every $\Cal A\subset\boldsymbol\pi^{-1}g$,
and every $\sigma,\theta\in\langle\Cal S\rangle$ one has\newline
$\Bbb P_{\boldsymbol\iota,g}(\Cal A|\Cal H_{g,\sigma,\theta})=
\Bbb P_{\boldsymbol\iota,h_{g,\sigma,\theta}}\{\varphi\in\boldsymbol\pi^{-1}h_{g,\sigma,\theta}:
\sigma\varphi\theta\in\Cal A\}$\rm.
\vskip3pt
\textit{Proof.}
It follows from the definitions that
$\boldsymbol\pi^{-1}(h_{g,\sigma,\theta})=
\{\varphi\in\langle\Cal S\rangle:\sigma\varphi\theta\in\boldsymbol\pi^{-1}g\}$ and\newline
 $\Cal A\cap\Cal H_{g,\sigma,\theta}=
\sigma{\cdot}\{\varphi\in\langle\Cal S\rangle:\sigma\varphi\theta\in\Cal A\}{\cdot}\theta$.

For $\varphi\in\langle\Cal S\rangle$ such that $\sigma\varphi\theta\in\Cal A$ we have
$\Bbb P_{\boldsymbol\iota,g}\{\sigma\varphi\theta\}=\Frac{\mu_\sigma\mu_\varphi\mu_\theta}{\s{Gr}_\mu(g)}$,
hence
$\Bbb P_{\boldsymbol\iota,g}(\Cal A\cap\Cal H_{g,\sigma,\theta})=
\Frac{\mu_\sigma{\cdot}\sum\{\mu_\varphi:\sigma\varphi\theta\in\Cal A\}{\cdot}\mu_\theta}{\s{Gr}_\mu(g)}=
\Frac{\mu_\sigma\s{Gr}_\mu(h_{g,\sigma,\theta})\mu_\theta}{\s{Gr}_\mu(g)}{\cdot}
\Bbb P_{\boldsymbol\iota,h_{g,\sigma,\theta}}\{\varphi\in\boldsymbol\pi^{-1}h_{g,\sigma,\theta}:
\sigma\varphi\theta\in\Cal A\}$.

Applying this to the set $\Cal A\leftrightharpoons\boldsymbol\pi^{-1}g$
we obtain
$\Bbb P_{\boldsymbol\iota,g}\Cal H_{g,\sigma,\theta}=
\Frac{\mu_\sigma\s{Gr}_\mu(h_{g,\sigma,\theta})\mu_\theta}{\s{Gr}_\mu(g)}$,\hfil\penalty-10000
$\Bbb P_{\boldsymbol\iota,g}(\Cal A|\Cal H_{g,\sigma,\theta}){=}
\Frac{\Bbb P_{\boldsymbol\iota,g}(\Cal A\cap\Cal H_{g,\sigma,\theta})}{\Bbb P_{\boldsymbol\iota,g}\Cal H_{g,\sigma,\theta}}=
\Bbb P_{\boldsymbol\iota,h_{g,\sigma,\theta}}\{\varphi\in\pi^{-1}h_{g,\sigma,\theta}:\sigma\varphi\theta\in\Cal A\}$.
\hfill$\square$
\section{Ancona properties}
\subsection{Weak Ancona's property}
For elements $x,z\in G$ define\newline
 $[x,z]\leftrightharpoons\{y\in G:|x\-y|+|y\-z|=|x\-z|\}$ the set of points belonging to the
\it word-geodesic intervals \rm between $x $ and $z$.

The following \it weak Ancona's property \rm is essential for what follows. It depends on parameter $\varepsilon\in(0,1)$.
\vskip3pt
$\s{WA}_\varepsilon:$ for every $x,z\in G,y\in[x,z]$
 one has
$\Bbb P_{x,z}\{\psi\in\s{Pth}(x,z):y{\in}\s{Im}\psi\}\geqslant 1-\varepsilon$.
\vskip3pt
The main result of \cite{Anc88} is the following
\vskip3pt
\textbf{Theorem} (A. Ancona \cite[Thm. 6.1]{Anc88}) \sl Let $(G,\mu)$ be an admissible pair such that the group $G$ is hyperbolic.
Then there exists $\varepsilon\in(0,1)$
such that $\s{WA}_\varepsilon$ holds for \rm$(G,\mu)$.
\vskip3pt
\textbf{Remark.} One can restate property $\s{WA}_\varepsilon$ in terms of the Green distance
$\displaystyle d_G(x,y)=-\ln {\s{Gr}_\mu(x,y)\over \s{Gr}_\mu(\iota, \iota)}$ as follows.
  By \cite[Lemma 1.13]{Wo00} one has
 {$\displaystyle \Bbb P_{x,z}\{\psi:y{\in}\s{Im}\psi\}={\s{Gr}_\mu(x,y){\cdot}\s{Gr}_\mu(y,z)\over \s{Gr}_\mu(x,z)}.$}\newline So $\s{WA}_\varepsilon$ is equivalent to the following inequality: \newline

\centerline{$\s{Gr}_\mu(x,z)\leqslant C\s{Gr}_\mu(x,y)\cdot \s{Gr}_\mu(y,z),\ y\in [x,z],\ C=\Frac 1{1-\varepsilon}.$}
\vskip3pt
 By taking logarithm of the both sides of it, we obtain the following equialent form of condition $\s{WA}_\varepsilon$:
 the Gromov product $(x\cdot_y z)$ with respect to the Green metric
for $y \in [x,z]$ is uniformly bounded:\vskip3pt
\centerline{$d_G(x,y)+d_G(y,z)\leqslant d_G(x,z)+D,$}
\vskip3pt
where $D$ not depending on $x, y$ and $z.$ 
\subsection{Thick Ancona's property} Our Theorem 3 is a weaker form of this conjecture.
We replace the property $\s{WA}_\varepsilon$
by the following stronger {\sl Thick Ancona's property} (briefly $\s{TA}$-property).
This property contains two numeric parameters $r$ and $\varepsilon$.
\vskip3pt
$\s{TA}_{\varepsilon,r}:$ For every $x,z\in G$ and every $y\in[x,z]$
 one has $\Bbb P_{x,z}\{\psi\in \s{Pth}(x,z):B_{\leqslant r}(y)\cap\s{Im}\psi\ne\varnothing\}>1-\varepsilon$.
\vskip3pt
Recall that $B_{\leqslant r}y$ denote the closed ball of radius $r$ centered at $y$.
\vskip3pt

Obviously, property $\s{WA}_\varepsilon$ coincides with property $\s{TA}_{\varepsilon,0}$.
On the other hand we have the following Lemma.

 \sc Lemma\sl.   The property $\s{TA}_{\varepsilon,r}$ implies
$\s{WA}_{\varepsilon'}$ where $\varepsilon'$ depends on $\varepsilon$ and $r$\rm.
\begin{proof} It follows from the argument given in the proof of \cite[Proposition 2.4]{GGPY21}.
 Indeed according to the Remark above, for $y\in[x,z]$ we have\newline
 $\Bbb P_{x,z}\{\psi\in \s{Pth}(x,z):B_{\leqslant r}(y)\cap\s{Im}\psi\ne\varnothing\}\leqslant\displaystyle\sum_{w\in B
_r(y)}{\s{Gr}_\mu(x,w){\cdot}\s{Gr}_\mu(w,z)\over\s{Gr}_\mu(x,z)}$.\newline
Thus, $\s{Gr}_\mu(x,y)<(1-\varepsilon)\displaystyle\sum_{w\in B_r(y)}\s{Gr}_\mu(x,w){\cdot}\s{Gr}_\mu(w,z)$.
 Since $|w\-y|\leqslant r$,
by Harnack	nequality \cite[(25.1)]{Wo00} we have
 \vskip3pt
 \centerline{$L^{-r}\s{Gr}_\mu(x,w)\leqslant\s{Gr}_\mu(x,y)\leqslant L^{r}\s{Gr}_\mu(x,w),$}
\noindent
where $L>1$ is a uniform constant. \vskip3pt Thus \vskip3pt
\centerline{$\s{Gr}_\mu(x,z)\leqslant (1-\varepsilon)L^{2r}|B_{\leqslant r}(y)|\s{Gr}_\mu(x,y){\cdot}\s{Gr}_\mu(y,z).$}
\vskip3pt
Here $|B_{\leqslant r}(y)|$ is the number of vertices in $B_{\leqslant r}(y)$.
Thus property $\s{WA}_{\varepsilon'}$ holds for\hfil\penalty-10000
 $\varepsilon'=1-(1-\varepsilon) L^{2r} |B_{\leqslant r}(y)|$.
\end{proof}
\vskip3pt
\begin{prop}
Let $G$ be a nonelementary hyperbolic group and let $\mu$ be an admissible measure on $G$.
Then for every $\varepsilon{\in}(0,1)$ there exists $r>0$ such that the property $\s{TA}_{\varepsilon,r}$   holds.
\end{prop}
\begin{proof} To use a result of \cite{GGPY21}
we briefly recall the description of the Floyd distance $\boldsymbol\delta_y^f(x,z)$
between vertices of a connected graph.

 Let $f:\mathbb N\to\mathbb R_{>0}$ be a non-increasing function
(called \textsl{Floyd function}) satisfying the conditions:\newline
$\displaystyle\ \sum _{n\in\mathbb N}f(n)<\infty$ and
$\Frac{f(n)}{f(n+1)}$ is bounded.

We  rescale the canonical distance by assigning to an edge $e$ of the Cayley graph the length $f(\s d(y,e))$.
 The Floyd length of a path is the sum of the  rescaled lengths of all its edges.
The Floyd distance $\boldsymbol\delta_y^f(x,z)$ is the  minimum of the lengths of paths between $x$ and $z$ (i.e. the shortest path distance).

The Cauchy completion $\overline G_f$
of the Cayley graph with respect to the Floyd distance,
is called \textit{Floyd completion}
(with respect to $f$ and $y$). It is compact since the graph has bounded degree,
and does not depend on the vertex $y$. The space
$\partial_f G=\overline G_f\setminus G$ is called \textit{Floyd boundary}.

Theorem 5.2. of \cite{GGPY21} claims that if $G$ is a finitely generated and $f$ is a Floyd function then\hfil\penalty-10000
for every $\varepsilon\in(0,1)$ and every $d>0$ there exists $r>0$ such that for every $x,y,z\in G$
for which $\boldsymbol\delta_y^f(x,z)>d$ one has
 $\Bbb P_{x,z}\{\psi:B_r(y)\cap\s{Im}\psi\ne\varnothing\}>1-\varepsilon$.

Note that in this Theorem the radius $r$ depends on two parameters $\varepsilon$ and $d$. To apply it, we need to show
that, in our case, the parameter $d$ is a uniform positive constant, i.e. once $G$ is hyperbolic and $y\in[x,z]$ then
$\boldsymbol\delta_y^f(x,z)>d>0$ where $d$  does not depend on $x,z,y$.

By Gromov's remark \cite[7.2M]{Gr87} there exists $\lambda\in]0,1[$ such that the Floyd compactification $\overline G_f$
of $G$ with respect to the scaling function $f(n)\leftrightharpoons\lambda^n$, coincides with
the Gromov compactification $G\sqcup\partial_\infty G$.

Suppose by contradiction that such uniform constant $d$ does not exist for $f$.
 Thus there exist sequences $x_n$, $y_n$, $z_n$ with
$y_n\in[x_n,z_n]$ such that $\boldsymbol\delta_{y_n}^f(x_n,z_n)\to 0$ as $n\to\infty$.

Without loss of generality we can assume that $y_n=\boldsymbol\iota$.
By compactness of $\overline G_f$, we can pass to subsequences
such that $x_n\to x_\infty$, $z_n\to z_\infty$ for some points $x_\infty$ and $z_\infty$ in the
Floyd boundary $\partial_fG$.
Since $\boldsymbol\delta_{y_n}^f(x_n,z_n)\to 0$ we have $x_\infty=z_\infty$.

By definition of the Gromov boundary $\partial_\infty G$ this implies that the Gromov product
$(x_n{\cdot}_{\boldsymbol\iota}z_n)$ tends to $\infty$.
On the other hand, since $\boldsymbol\iota\in[x_n,z_n]$ we have $(x_n{\cdot}_{\boldsymbol\iota}z)=0$.
This contradiction completes the proof.
\end{proof}
\section{Rate of transience} We need the following well-known characterizatoin
of amenable groups \cite{Kes59}:\hfil\penalty-10000
 a finitely generated group $G$ is amenable if and only if, for some (=for any)
admissible measure $\mu$ on $G$, the \textit{spectral radius} $\boldsymbol\rho(G,\mu)\leftrightharpoons
\underset{n\rightarrow\infty}{\s{lim\,sup}}(\mu_{\boldsymbol\iota}^{(n)})^{1/n}=1$.

If $G$ is non-amenable then $\boldsymbol\rho(G,\mu)<1$ for every admissible $\mu$.
Thus in the class of transient groups the amenable ones can be characterized as   groups with
the slowest transience.
Every non-amenable group is transient.

The following fact is well-known. Since it is sometimes formulated in more general settings,
for the reader's convenience we give a proof.
\begin{prop}
Let $(G,\mu)$ be an admissible pair with $G$ nonamenable and $\rho\leftrightharpoons\boldsymbol\rho(G,\mu)$
be the spectral radius.
 Then, for every $g\in G$ and $m\geqslant|g|$, one has\newline
$\s{NA}:$\quad $\Bbb P_{x,y}\{\psi\in\pi^{-1}(x^{-1}y):\ell_\psi\geqslant m\}\leqslant
\rho^{m-D|x-y|-N}$ where
$D\leftrightharpoons\s{log}_\rho\s{min}(\mu|_{\Cal S})$,\newline
$N\leftrightharpoons\s{log}_\rho(\s{Gr}_\mu(\boldsymbol\iota)(1-\rho))\geqslant0$.
\end{prop}

We call $\s{NA}$ the \it nonamenability inequality\rm.
\begin{proof}
By \cite[Lemma 1.9]{Wo00}, $\mu^{(n)}(x,x)\leqslant\rho^n$.
Hence\hfil\penalty-10000
$(1)$\quad$\s{Gr}_\mu(x,x)=
\sum_{n\in\Bbb N}\mu^{(n)}(x,x)\leqslant\sum_{n\in\Bbb N}\rho^n=\Frac1{1-\rho}$;\hfil\penalty-10000
$\s{Gr}_\mu(x,x)(1-\rho)\leqslant1$;\hfil\penalty-10000
$(2)$\quad$N=\s{log}_\rho(\s{Gr}_\mu(x,x)(1-\rho))\geqslant0$.

For $x,z\in G$, by symmetry of $\mu$, $\mu^{(n)}(x,z)=\mu^{(n)}(z,x)$.
Hence\hfil\penalty-10000
$(\mu^{(n)}(x,z))^2=\mu^{(n)}(x,z)\mu^{(n)}(z,x)\leqslant\mu^{(2n)}(x,x)\leqslant\rho^{2n}$,
and\hfil\penalty-10000
$(3)$\quad$\mu^{(n)}(x,z)\leqslant\rho^n$.

On the other hand, let\hfil\penalty-10000
$(4)$\quad$\lambda{\leftrightharpoons}\s{min}\{\mu_s:s{\in}\Cal S\}$.\hfil\penalty-10000
If $|x\-z|=|x\-y|-1$ and $|z\-y|=1$, then
 $\mu^{(n)}(x,y)\overset{(4)}\geqslant \mu^{(n-1)}(x,z){\cdot}\lambda$.\hfil\penalty-10000
By iterating the last inequality we get:\hfil\penalty-10000
$(5)$\quad $\mu^{(n)}(x,y)\geqslant\lambda^d\mu^{(n-d)}(x,x)$, where $d\leftrightharpoons|x$-$y|$.
From here we have\hfil\penalty-10000
$(6)$\quad$\s{Gr}_\mu(x,y){=}\sum_{i=d}^\infty\mu^{(i)}(x,y)\overset{(5)}\geqslant
\lambda^d\sum_{i=0}^\infty\mu^{(i)}(x,x)=\lambda^d\s{Gr}_\mu(x,x)$
(by the Harnack inequality).

Using the estimates $(3)$ and $(6)$, for $m\geqslant d$ we obtain:\hfil\penalty-10000
$\Bbb P_{x,y}\{\psi\in\pi^{-1}(x^{-1}y),\,\ell_\psi\geqslant m\}=$\hfil\penalty-10000
${=}\Frac1{\s{Gr}_\mu(x,y)}\sum_{i=m}^\infty\mu^{(i)}(x,y)\overset{(3)}\leqslant
\Frac1{\s{Gr}_\mu(x,y)}\sum_{i=m}^\infty\rho^i=$\newline$=
\Frac{\rho^m}{\s{Gr}_\mu(x,y)\cdot (1{-}\rho)}\overset{(6)}\leqslant
\Frac{\rho^m}{\s{Gr}_\mu(x,x)\cdot \lambda^d\cdot (1{-}\rho)}=\rho^{m-Dd-N}$,
where
$D{\leftrightharpoons}\s{log}_\rho\lambda$ and $N$ is from (2).
\end{proof}
\section{Proof of Theorems 1 and 2}
\subsection{Bypassing a ball}
Let $M$ be a geodesic metric space and let $I=[x,z]$ be a closed geodesic interval in $M$.
Denote by $B_I$ the open ball of radius $r_I\leftrightharpoons\ell_I/2$ centered at the middle point $y$ of $I$.
Denote by $\Delta_I$ the set of the rectifiable curves $\gamma$ in $M$
that join the endpoints of $I$ and miss $B_I$.

Denote by $\pi(I)$ the number $\Frac{\s{inf}\{\ell_\gamma:\gamma\in\Delta_I\}}{r_I}$
 (according to a
natural convention, $\pi(I)=\infty$ if $\Delta_I=\varnothing$).

Note that, if $M$ is an Euclidean space of dimension
 at least two then $\pi(I)=\pi=3.14159\dots$.

Denote $\pi(M)\leftrightharpoons\underset{\ell I\rightarrow\infty}{\s{lim\,inf}}\pi(I)$
 (by the same convention,
if $M$ is bounded then $\pi(M)=\infty$).

It is well-known, see for instance \cite[7.1A]{Gr87}, that,
 if all geodesic triangles in $M$
are $\delta$-thin for a positive
 $\delta$, then $\pi(I)\geqslant\delta(2^{r_I/\delta}-2)$.
In particular,
 if $M$ is hyperbolic
 then $\pi(M)=\infty$.

\begin{prop}{\rm (Theorem 1)}
If $\Gamma$ is a connected graph and
 $\pi(\Gamma)>10$ then $\Gamma$ is hyperbolic.
\end{prop}
\begin{proof}
Assume that $\Gamma$ is not hyperbolic. It suffices to find
 a sequence $\{I_k:k\in\Bbb N\}$
such that $\ell(I_k)\to\infty$ and $\pi(I_k)\leqslant10$.

By Theorem P (see subsection 2.4) the width of geodesic bigons
 in $\Gamma$ is unbounded.

Instead of the ``uniform'' width function $\s w_\beta$ consider the Hausdorff width function\newline
$\s v_\beta\leftrightharpoons\s{inf}\{c:$ each side of the bigon $\beta$ is contained in the $c$-neighbourhood
of the other side of $\beta\}$.

Let $\beta$ be a geodesic bigon $\beta$ such that $\s v_\beta=c$.
Without loss of generality we can assume that there exists a point $p$ on the side $\beta_0$
such that $\s{dist}(p,\s{Im}\beta_1)=c$. Let $p$ be such point (see Figure 2).
Let $I_c$ be $\{x\in\s{Im}\beta_0:|p\-x|\leqslant c\}$.
Since $\partial\beta\subset\s{Im}\beta_1$ we have $\s{dist}(p,\partial\beta\geqslant c)$.
Hence $\ell(I_c)=2c$. We will show that $\pi(I_c)\leqslant10$.\par\noindent
\begin{picture}(0,0)(20,80)
\put(40,0){
\pdfximage width 400pt{\mpPath 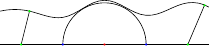}\pdfrefximage\pdflastximage}
\put(244,-2){\makebox(0,0)[ct]{$p$}}
\put(244,-16){\makebox(0,0)[cc]{Figure 2}}
\put(244,50){\makebox(0,0)[ct]{$B_{I_c}$}}
\put(254,88){\makebox(0,0)[cb]{$J$}}
\put(166,-2){\makebox(0,0)[ct]{$p_-$}}
\put(326,-2){\makebox(0,0)[ct]{$p_+$}}
\put(86,45){\makebox(0,0)[ct]{$I_-$}}
\put(86,-2){\makebox(0,0)[ct]{$q_-$}}
\put(406,-2){\makebox(0,0)[ct]{$q_+$}}
\put(412,45){\makebox(0,0)[ct]{$I_+$}}
\put(103,68){\makebox(0,0)[cb]{$r_-$}}
\put(443,78){\makebox(0,0)[cb]{$r_+$}}
\put(30,1){\makebox(0,0)[cl]{$\beta_0$}}
\put(32,65){\makebox(0,0)[cl]{$\beta_1$}}
\end{picture}
\vskip105pt
Let $p_-$ and $p_+$ denote the endpoints of $I_c$.
Suppose that there exist points $q_-,q_+\in\s{Im}\beta_0$
such that $p$ is between them and $|q_--p|=|q_+-p|=2c$ (see Figure 2).
Let $r_-$ be the point in $\s{Im}\beta_1$ such that $|q_--r_-|=\s{dist}(q_-,\s{Im}\beta_1)$.
Similarly, let $r_+$ be the point in $\s{Im}\beta_1$ such that $|q_+-r_+|=\s{dist}(q_+,\s{Im}\beta_1)$.
Let $I_-$ be a geodesic segment such that $\partial I_=\{q_-,r_-\}$
and let $I_+$ be a geodesic segment such that $\partial I_+=\{q_+,r_+\}$.

By the choice of the point $p$ we have $\ell(I_-)\leqslant c$ and $\ell(I_+)\leqslant c$.
Since $|q_--p|=2c$ we have $I_-\cap B_{I_c}=\varnothing$. Similarly, $I_+\cap B_{I_c}=\varnothing$.

Considering the broken line $r_--q_--p-q_+-r_+$ formed by $I_{\pm}$ and the part $[q_-,q_+]$ of $\beta_0$
we obtain: $|r_--r_+|\leqslant6c$.

Let $J$ be the part $[r_-,r_+]$ of $\beta_1$. Its length is at most $6c$.
By construction the broken line $p_--q_--r_-r_+-q_+-q$ whose part $[r_-,r_+]$ is $J$ does not
intersect $B_{I_c}$ and has length at most $10c$. Hence $\pi(I_c)\leqslant10$.

Suppose that $q_-$ does not exist and $q_+$ esxists. Then the distance from $p$ to $\beta_0(0)=\beta_1(0)$
is less than $2c$. In this case we define $q_-\leftrightharpoons r_-\leftrightharpoons\beta_0(0)
=\beta_1(0)$.
A similar observation shows that in this case $\pi(I_c)\leqslant 8$.
Analogously, if $q_-$ exists and $q_+$ does not exist then $\pi(I_c)\leqslant 8c$.
If neither $q_-$ nor $q_+$ exist then, by a similar argument, $\pi(I_c)\leqslant 6$.

Since $c$ is unbounded we have $\pi(M)\leqslant10$ contradicting the assumption.
\end{proof}
\subsection{Application to random walks}
Let $(G,\mu)$ be an admissible pair such. We keep the notation of the previous subsection.
For a geodesic word $\gamma\in\langle\Cal S\rangle$
denote by $I_\gamma$ the geodesic segment from $\boldsymbol\iota$ to $\boldsymbol\pi_\gamma$
in the Cayley graph $\Gamma\leftrightharpoons\s{Ca}(G,\Cal S)$.

We suppose that the group $G$ is transient that is not a virtually abelian of rank at most two.
\begin{prop} {\rm (Theorem 2)}
A transient group $G$ is hyperbolic if and only if there exists
 $\varepsilon<\varepsilon_0\leftrightharpoons\s{min}(\mu|_{\Cal S})$
such that, for every sufficiently long geodesc word $\gamma$,
$\Bbb P_{\boldsymbol\iota,\boldsymbol\pi(\gamma)}\{\psi\in\boldsymbol\pi^{-1}(\boldsymbol\pi(\gamma)):
\psi^{-1}B({I_\gamma})=\varnothing\}\leqslant\varepsilon^{5\ell(\gamma)}$.
\end{prop}
\begin{proof}
We first prove the sufficiency.

Suppose by contradiction that the condition holds, but $G$ is not hyperbolic.
Then, by Proposition 3, $\pi(\Gamma)\leqslant10$.
Thus there exists a sequence $\gamma_k$ ($k\in\Bbb N$) of geodesic words such that $\ell(\gamma_k)\to\infty$
and\newline
$(1)$\quad $\sigma\leftrightharpoons\underset{k\rightarrow\infty}{\s{lim}}\pi(I_{\gamma_k})\leqslant10$.\newline
For $k\in\Bbb N$, denote $g_k\leftrightharpoons\boldsymbol\pi(\gamma_k)$, $r_k\leftrightharpoons|g_k|/2$.
Let $\psi_k\in\boldsymbol\pi^{-1}g_k$ be such that
$\psi_k^{-1}B(I_{\gamma_k})=\varnothing$.

The $\mu$-weight $\mu_{\psi_k}$ of $\psi_k$ is at least $\varepsilon_0^{\ell(\psi_k)}$.
Hence\hfil\penalty-10000
$(2)$\quad$\Bbb P_{\boldsymbol\iota,g_k}\{\psi_k\}\geqslant\Frac{\varepsilon_0^{\ell\psi_k}}M$
where $M\leftrightharpoons\s{Gr}_\mu(\boldsymbol\iota)=\s{max\{Gr}_\mu(g):g\in G\}$.\newline
By the assumption, for sufficiently big $k$, one has
$\Bbb P_{\boldsymbol\iota,g_k}\{\psi\in\boldsymbol\pi^{-1}g_k:\psi^{-1}B(I_{\gamma_k})=\varnothing\}\leqslant\varepsilon^{10r_k}$.\newline
By $(2)$, for these values of $k$, one has\newline
$(3)$\quad $\Frac{\varepsilon_0^{\ell\psi_k}}M\leqslant\varepsilon^{10r_k}$.

Denote $\delta_k\leftrightharpoons\Frac{\ell\psi_k}{r_k}-10$.
By $(1)$, if $\sigma<10$ then $\delta_k<0$ for $k\gg0$ and if $\sigma=10$ then
$\delta_k\overset{k\to\infty}\longrightarrow0$. Thus\newline
$(4)$ $\forall\delta>0\exists k_\delta\in\Bbb N\,\forall k\geqslant k_\delta:\delta_k\leqslant\delta$.

Substituting to $(3)$ one has:
$\Frac{\varepsilon_0^{(10+\delta_k)r_k}}M\leqslant\varepsilon^{10r_k}$, and, consequtely:\newline
$\Frac{\varepsilon_0^{(10+\delta_k)r_k}}{\varepsilon^{10r_k}}\leqslant M$,
$\left(\Frac{\varepsilon_0^{10+\delta_k}}{\varepsilon^{10}}\right)^{r_k}\leqslant M$,
$\left((\varepsilon_0/\varepsilon)^{10}{\cdot}\varepsilon_0^{\delta_k}\right)^{r_k}\leqslant M$.

Set $\delta\leftrightharpoons(\s{log}_{\varepsilon_0}\varepsilon)-1>0$.
By $(4)$, for sufficiently big $k$, one has $\delta_k\leqslant\delta$, and hence
$\varepsilon_0^{1+\delta_k}\geqslant\varepsilon$.
So, $\varepsilon_0^{\delta_k}{\geqslant}\varepsilon/\varepsilon_0$ and
$(\varepsilon_0/\varepsilon)^{10}{\cdot}\varepsilon_0^{\delta_k}\geqslant(\varepsilon_0/\varepsilon)^9$.
Therefore $(\varepsilon_0/\varepsilon)^{9r_k}\leqslant M$.
The left-hand side of this inequality tends to infinity while the right-hand side is constant.
This contradiction implies that $\pi(\Gamma)>10$ and the group $G$ is hyperbolic by Proposition 5.

We now prove the necessity.

By Proposition 1,
for every nonelementary hyperbolic group $G$ and for every $\varepsilon\in(0,1)$
 there exists $r\geqslant0$ such that property $\s{TA}_{r,\varepsilon}$ holds.
We will show that the upper bound $\varepsilon(r)$ for the probability that a random path bypasses a ball of radius $r$
can be chosen to have superexponential decay:\hfil\penalty-10000
 $\underset{r\rightarrow\infty}{\s{lim}}\Frac{\s{ln}(\varepsilon(r))}r=-\infty$.

This estimate come from the proof of the  main result of \cite{GGPY21}.
 Indeed by formula $(24)$ of \cite[page 782]{GGPY21} $\varepsilon(r)$
can be taken to be equal to $\displaystyle\sum_{i=0}^\infty\phi^{h(\theta^ir)-2\theta^iDr}$
where $\phi<1$, $\theta>1$ and $h(r)$ is a function satisfying $h(r)/r\to\infty$ as $r\to\infty$.
So, for every $C>0$ and sufficiently big $r$ one has $h(\theta^ir)-2\theta^iDr>Cr(i+1)$ for all $i$.
Therefore $\varepsilon(r)<\phi^{Cr}\sum_{i\geqslant0}\phi^{iCr}=\Frac{\phi^{Cr}}{1-\phi^{Cr}}<2\phi^{Cr}$ ($r\gg0$).
We have $2\phi^{Cr}<\varepsilon^{10r}$ for $C>10{\cdot}\s{log}_\phi({\varepsilon\over2})$.
Theorem 2 is proved.\end{proof}

\section{A preliminary result on width functions}
\subsection{Proper functions and suitable intervals}
For a metric space $M$ and a bigon $\beta{\in}\s{gb}M$ the width function $f\leftrightharpoons\s w_\beta$
has the following properties:\newline
1: $f$ is defined on a closed interval $\s{dom}(f)\leftrightharpoons I$ of finite length;\newline
2: $f\geqslant0$;\newline
3: $f|_{\partial I}=0$;\newline
4: $f$ is 2-Lipschitz: $|f(x){-}f(y)|\leqslant2{\cdot}|x-y|$ (which follows from the triangle inequality).

Let us call a real function \it proper \rm if it satisfies the conditions 1--4.
Thus every width function of a geodesic bigon is proper.

\textbf{Definition.} Let $k$ be a positive number.
For a proper function $f$ with $\s{dom}(f)=I$,
a closed interval $J\subset I$ is called $k$\it-suitable for \rm $f$
it the following condition holds:\hfil\penalty-10000
$(1)$\quad$\Frac{\ell_J}{2k}\leqslant\s{min}(f|_J)=f|_{\partial J}\leqslant\Frac{\ell_J}k$.

The $k$\it-height of \rm$f$ is the number $\s h_k(f)\leftrightharpoons\s{sup\{min}(f|_J):J$ is $k$-suitable for $f\}$.

\begin{prop}
For every $k>0$ and every proper function $f$ one has $f\leqslant(2k+1)\s h_k(f)$.
\end{prop}

Remark. Applying to the width of functions $\s w_\beta$ of bigons this proposition gives: if
the width of bigons is unbounded then, for every $k>0$, the function $\beta\mapsto\s h_k(\s w_\beta)$
is undbounded too.

\begin{proof}
Let $I\leftrightharpoons\s{dom}(f)$.

On the set $\Cal C\leftrightharpoons\{$the closed nonempty intervals$\subset I\}$,
we consider the distance
 $$\varrho(R,S)\leftrightharpoons\s{max\{|min}R-\s{min}S|,|\s{max}R-\s{max}S|\}.$$
Thus $(\Cal C,\varrho)$ is a compact metric space (homeomorphic to a closed triangle in the plane).
The maps $\Cal C\ni R\mapsto\s{min}R$ and $\Cal C\ni R\mapsto\s{max}R$ are 1-Lipschitz and hence continuous.
Hence the function $R\mapsto\ell_R$ is continuous.
\vskip3pt
{\sc Lemma} 1. \textit{If $R,S\in\Cal C$ and $p\in R$ then there exists $q\in S$ such that}
$|p-q|\leqslant\varrho(R,S)$.
\vskip3pt
\noindent\textit{Proof.} If $p\in S$ then we can take $q\leftrightharpoons p$.
If $p\notin S$ then $p$ is from the left or from the right side of $S$.
 In the first case, for $q\leftrightharpoons\s{min}S$
one has $p-q\leqslant\s{min}R-\s{min}S\leqslant\varrho(R,S)$. Similarly, in the second case
one can take $q\leftrightharpoons\s{max}S$.\hfill$\square$
\vskip3pt
{\sc Lemma} 2. \textit{The function $g:\Cal C\ni R\mapsto\s{min}(f|_R)$ is $2$-Lipschitz}.
\vskip3pt\noindent
\textit{Proof.} Let $R,S\in\Cal C$. Denote $e\leftrightharpoons\varrho(R,S)$.
Let $p\in R$ be such that $f(p)=g(R)$. By Lemma 1 there exists $q\in S$ such that $|p-q|\leqslant e$.
Since $f$ is 2-Lipschitz, $|f(p)-f(q)|\leqslant2e$. Hence
$g_S=\s{min}(f|_S)\leqslant f(q)\leqslant f(p)+2e=g_R+2e$.
By symmetry, $g_R\leqslant g_S+2e$. It follows that
$|g_R-g_S|\leqslant 2e=2{\cdot}\varrho(R,S)$.\hfil$\square$
\vskip3pt
Since $k>0$, it follows from Lemma 2 that the set $\Cal F\leftrightharpoons\{R\in\Cal C:g_R=f|_{\partial R},
\ell_R\geqslant k{\cdot}g_R\}$ is closed in $\Cal C$ and so is compact.
It is nonempty since $I\in\Cal F$.

Let $M\leftrightharpoons\s{max}(g|_{\Cal F})$.
If $M{=}0$ then $f=0$ and the proposition holds.
Suppose that $M>0$.

Divide the set $\Cal D\leftrightharpoons\{$the connected components of $I\setminus f^{-1}M\}$ as follows:\newline
$\Cal D_-\leftrightharpoons\{P\in\Cal D:f|_P<M\}$;\hfil\penalty-10000
$\Cal D_0\leftrightharpoons\{P\in\Cal D:f|_P>M,\ell_P<kM\}$;\hfil\penalty-10000
$\Cal D_1\leftrightharpoons\{P{\in}\Cal D:f|_P>M,kM\leqslant\ell_P\leqslant2kM\}$;\hfil\penalty-10000
$\Cal D_2\leftrightharpoons\{P\in\Cal D:f|_P>M,\ell_P>2kM\}$.
\vskip3pt
{\sc Lemma} 3. $\Cal D_2=\varnothing$.
\vskip3pt\noindent
\textit{Proof.}
Let $Q\in\Cal D_2$. Since the functions $\Cal C\ni R\mapsto\ell_R$ and $g$
are continuous and $\ell_{\overline Q}>2kM$, $g_{\overline Q}=M$,
there exists $R{\in}\Cal C$ such that $R\subset Q$, $\ell_R>2kM$ and $M_1\leftrightharpoons g_R<2M$
 (as $M\ne0$).
Since $f|_Q>M$ we have $M_1>M$.

Let $Q=]a,b[$ and $R=[a_1,b_1]$. We have $f(a)=M$, $f(a_1)\geqslant M_1$.
By continuity of $f$ the set $[a,a_1]\cap f^{-1}M_1$ is closed and nonempty.
Let $a_2\leftrightharpoons\s{max}([a,a_1]\cap f^{-1}M_1)$.
We have $f^{-1}M_1\cap]a_2,a_1[=\varnothing$ and $f(a_2)=M_1$.

In the same way we find a point $b_2\in[b_1,b]$ such that $f(b_2)=M_1$ and $f|_{[b_1,b_2]}\geqslant M_1$.
Let $S\leftrightharpoons[a_2,b_2]$. We have $g_S=M_1$ and
$\ell_S\geqslant\ell_R>2kM\geqslant kM_1=k{\cdot}g_S$.
So $S\in\Cal F$ and we have a contradiction with the definition of $M$.\hfill$\square$
\vskip3pt
{\sc Lemma} 4. $M\leqslant\s h_k(f)$.\newline
\textit{Proof.}
It suffices to find a $k$-suitable interval $S$ with $g_S=M$.
For every $Q\in\Cal D_1$ the interval $\overline Q$ is $k$-suitable for $f$.
Hence we can assume that $\Cal D_1=\varnothing$.

Let $K\in\Cal F$ and $g_K=M$. By Lemma 3 and our assumption, all $Q\in\Cal D$ such that
$Q\subset K$ belong to $\Cal D_0$. Let $a\leftrightharpoons\s{min}K$, $b\leftrightharpoons a+kM$
and $R\leftrightharpoons[a,b]$. Since $K\in\Cal F$ we have $b\in K$.
If $f(b)=M$ then $R$ is $k$-suitable for $f$.

If $f(b)>M$ then let $Q\in\Cal D_0$ be such that $b\in Q$. The set $S\leftrightharpoons R\cup\overline Q$
is a closed interval with $g_S=M$, $f|_{\partial S}=M$ and $\ell_S\leqslant k+\ell_Q<2k$.
Hence $S$ is $k$-suitable for $f$.\hfill$\square$
\vskip3pt
Remark. Every $k$-suitable interval belongs to $\mathcal F$ thus $M\geqslant\s h_k(f)$.
So by Lemma we have $M=\s h_k(f)$.
\vskip3pt
It remains to estimate $f$. Let $a\in I$ and $f(a)=\s{max}f$.
We can assume that $f(a)>M$. Let $Q\in\Cal D$ be such that $a\in Q$. By Lemma 3, $Q\in\Cal D_0\cup\Cal D_1$
hence $\ell Q{\leqslant}2kM$. Let $p$ be the endpoint of $Q$ for which $|a-p|\leqslant\ell_Q/2\leqslant kM$
and $f(p)=M$.
By 2-Lipschitz condition $f(a)-f(p)\leqslant2|a-p|\leqslant2kM$. Thus
$f(a)\leqslant M+2kM=(2k+1)M$. It follows from Lemma 4 that $f\leqslant(2k+1)\s h_k(f)$ as claimed.\end{proof}
\section{Proof of Theorem 3}
The goal of this section is to prove Theorem 3.
In fact we prove a stonger Proposition 6, where the nonamenability assumption is replaced
by weaker conditions $\s A$ and $\s B$.
These conditions contain certain parameters that are included in
the estimate of the upper bound $\varepsilon_0$ for $\varepsilon$.
In particular the conditions are satisfied when the group $G$ is non-amenable.
\subsection{Conditions on $(G,\mu)$}
The first condition $\s A$ contains positive numbers $A$ and $a<1$ as parameters.
 It seems that many groups satisfy
$\s A$ for some $A,a$.
\vskip3pt\noindent
$\s{A=A}_{A,a}:\forall c\geqslant1:|g|\leqslant c\Rightarrow\mathbb P_{\boldsymbol\iota,g}\{\psi:\ell_\psi\geqslant Ac\}\leqslant a$.
\vskip3pt
\sc Lemma 1\sl. Every non-amenable admissible pair $(G,\mu)$ satisfies $\s A_{A,a}$ for some $A$ and \rm$a<1$.
\vskip3pt\noindent
\textit{Proof.} Let $\rho,D,N$ be the numbers from the proof of Proposition 2.
Let $A=D+N+1$. By $\s{NA}$, for every $g\in G$ one has
$\Bbb P_{\boldsymbol\iota,g}\{\psi:\ell_\psi\geqslant Ac|\}\leqslant$\newline$
\leqslant\rho^{Ac-D|g|{-}N}=
\rho^{(A-D)c+D(c-|g|)-N}=
\rho^{(N+1)c+D(c-|g|)-N}\leqslant\rho^{1+D(c-|g|)}\leqslant\rho<1$.\hfill$\square$
\vskip3pt

The second condition is more restrictive. It includes the hypothese that $\s A_{A,a}$ holds for
some particular $(A,a)$ and depends on two more positive parameters $B$ and $b$:
\vskip3pt\noindent
$\s{B=B}_{A,a,B,b}:\s A_{A,a}$ holds, $b<1-a$, and $\forall c\geqslant1,|g|\leqslant c\Rightarrow
\Bbb P_{\boldsymbol\iota,g}\{\psi:\ell_\psi\geqslant Bc\}\leqslant\Frac b{2(AB+B+2)}$.
\vskip3pt\noindent
\sc Lemma 2\sl. For every non-amenable admissible pair $(G,\mu)$ and every $\delta>0$ there
exist $B_\delta$ such that, for every $c\geqslant1$:
$|g|\leqslant c\Rightarrow\Bbb P_{\boldsymbol\iota,g}\{\psi:\ell_\psi\geqslant Bc\}<\Frac\delta B$ for every
$B\geqslant B_\delta$\rm.
\vskip3pt\noindent
\begin{proof} Let $\rho,D,N$ be as above. The function $B\mapsto B\rho^{B-D-N}$ tends to $0$ as $B\to\infty$.
Hence, for some $B_\delta>D$ and for every $B\geqslant B_\delta$
 one has $B{\cdot}\rho^{B-D-N}<\delta$. It follows from the inequality $\s{NA}$ that
$\Bbb P_{\boldsymbol\iota,g}\{\psi:\ell_\psi\geqslant Bc\}\leqslant
\rho^{Bc-D|g|-N}=\rho^{(B-D)c+D(c-|g|)-N}\leqslant\rho^{B-D-N}<\delta/B$.
\end{proof}
\vskip3pt
\sc Corollary\sl. For every non-amenable admissible pair $(G,\mu)$ there exist
$A,a<1,B,b<1-a$ such that $\s B_{A,a,B,b}$ holds\rm.
\vskip3pt\noindent
\textit{Proof.} By Lemma 1, there exist $A,a<1$ such that $\s A_{A,a}$ holds for $(G,\mu)$.
Let $0<b<1-a$ and $\delta\leftrightharpoons\Frac b{2(A+2)}$ and $B_\delta$ be from Lemma 2.
 If $B\geqslant B_\delta$ and $B>2$
then
for every $c\geqslant1$ such that $|g|\leqslant c$ one has
 $\Bbb P_{\boldsymbol\iota,g}\{\psi:\ell_\psi\geqslant Bc\}<\Frac b{2(AB+2B)}<
\Frac b{2(AB+B+2)}$.\hfill$\square$
\subsection{Main proposition} The remaining part of the section is devoted to the proof of
the following
\begin{prop}
If an admissible pair $(G,\mu)$ satisfies $\s B_{A,a,B,b}$ for some $A,a<1$, $b<1-a$ and
$\s{TA}_{\varepsilon,r}$ for some $r\geqslant 0$
and some\newline
$(1)$\quad $\varepsilon<\varepsilon_0\leftrightharpoons\Frac{1-a-b}{2(AB+B+3)}$\newline
 then $G$ is hyperbolic.
\end{prop}

Remark. The constant $\varepsilon_0$ in $(1)$ is the critical value mentioned in the Introduction.

We breifly describe the plan of the proof. Until the end of the proof we fix the values $A,a,B,b$ of parameters.
Assuming that $G$ is not hyperbolic, we deduce from Theorem P that the width of geodesic bigons is unbounded.
By Proposition 5, the functions $\s{gb(Ca}(G,\Cal S))\ni\beta\mapsto\s h_k(\s w_\beta)$ are unbounded
for every $k$. We choose a particular value for $k$ that depends on the parameters.
Considering a geodesic bigon $\beta$ with $\s h_k(\s w_\beta)=M$, we will obtain a contradiction
with the definition of the Green probability spaces $\boldsymbol\pi^{-1}(g)$ ($g\in G$).
It will follow once $M$ is sufficiently big, namely if it satisfies the inequalities
$(6)$, $(7)$, $(12)$, and $(15)$ below.
\subsection{The parameters $n$ and $k$}
Let $n$ denote the minimal integer that satisfies the inequality\newline
$n>AB+B+2$. Thus\newline
$(2)$\quad$AB+B+2<n\leqslant AB+B+3$.\newline
By $(1)$ this implies:\newline
$(3)$\quad$2n\varepsilon<1-a-b$.\newline
By $(2)$ one has $A+1<\Frac{n-2}B$, hence there exists $k\in\Bbb R$ such that\newline
$(4)$\quad$\Frac{A+1}2<k<\Frac{n-2}{2B}$.

Until the end of the proof we fix the values $n$ and $k$ satisfying $(2)$ and $(4)$.
\subsection{Paradoxical bigon}
Let $\beta=(\beta_0,\beta_1)$ be a geodesic bigon, $f\leftrightharpoons\s w_\beta$,
$M\leftrightharpoons\s h_k(f)$
and an interval $J=[p,q]$ of length $L=q-p$, $k$-suitable for $f$ (see the previous subsection). Thus\newline
$(5)$\quad$\Frac L{2k}\leqslant M\leftrightharpoons\s{min}(f|_J)=f|_{\partial J}\leqslant\Frac Lk$.\hfil\penalty-10000
We say that an interval $J$ is paradoxical if $M$ satisfies all our requirements.
We say that a bigon is paradoxical if its domain of definition contains a paradoxical interval.

Until the end of the proof we fix a paradoxical bigon $\beta$ and a paradoxcial interval $J$.

We require $M$ to satisfy the following inequalities:\newline
$(6)$ $M\geqslant\Frac{(2r+1)(n-1)}k$ and\newline
$(7)$ $M\geqslant 2r+1$.

It follows from $(2)$ that $n>2$.
Let $T\subset J$ be the set of cardinality $n$ that divides $J$
into $n-1$ pairwise congruent intervals. Thus, $p=\s{min}T$, $q=\s{max}T$.
The distance between two adjacent numbers in $T$ is\newline
$(8)$\quad$d\leftrightharpoons\Frac{q-p}{n-1}\overset{(5),(6)}\geqslant 2r+1$.

On each side $\beta_\iota$ of $\beta$ we consider the
closed balls $B_{\iota,t}\leftrightharpoons B_r(x_{\iota,t})$ centered at
$x_{\iota,t}\leftrightharpoons\beta_\iota(t)$
($t\in T$)
of radius $r$  where $r$ comes from the condition $\s{TA}_{\varepsilon,r}$ (see Figure 3).
We call this $2n$ balls $B_{i,\iota}$ the \it paradoxical balls\rm.

\noindent
\begin{picture}(10,130)(-40,-10)
\put(-14,-70){
\pdfximage{\mpPath 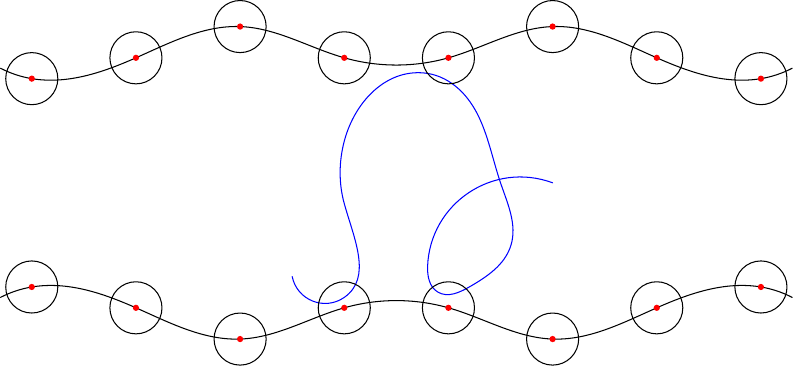}\pdfrefximage\pdflastximage}
\put(5,18){\vector(0,1){45}}
\put(5,18){\vector(0,-1){45}}
\put(4,18){\makebox(0,0)[rc]{$M$}}
\put(356,18){\vector(0,1){45}}
\put(356,18){\vector(0,-1){45}}
\put(358,18){\makebox(0,0)[lc]{$M$}}
	\put(27,18){\vector(1,0){325}}
	\put(270,18){\vector(-1,0){262}}
\put(178,16){\makebox(0,0)[ct]{$L$}}
\put(178,-70){\makebox(0,0)[cc]{Figure 3}}
\put(5,86){\makebox(0,0)[cc]{$B_{1,p}$}}
\put(5,-51){\makebox(0,0)[cc]{$B_{0,p}$}}
\put(55,97){\makebox(0,0)[cc]{$B_{1,p+d}$}}
\put(54,-60){\makebox(0,0)[cc]{$B_{0,p+d}$}}
\end{picture}
\vskip75pt
It follows from $(7)$ and $(8)$ that the balls $B_{\iota,t}$ are pairwise disjoint.
\subsection{Classification of paths} Let $(x,z)\leftrightharpoons\partial\beta$.
We consider the words in $Q\leftrightharpoons\boldsymbol\pi^{-1}(x^{-1}z)$ as paths from $x$ to $z$ as it is explained in Subsection 2.6.
We keep the notation of Subsection 2.9 for the associated Green probability spaces.

A path $\psi\in Q$ is called \it missing \rm if it does not intersect at least one of the paradoxical
balls $B_{\iota,t}$.

We have\hfil\penalty-10000
$(9)$\quad  $\mathbb P_{\partial\beta}Q_{\fam0missing}\leqslant
\displaystyle\sum_{\iota\in\{0,1\},t\in T}\mathbb P_{\partial\beta}\{\psi:(x,\psi)$ misses $B_{\iota,t}\}
\overset{\s{TA}_{\varepsilon,r}}\leqslant2n\varepsilon\overset{(3)}\leqslant 1-a-b$.

A path $\psi\in Q$ is called \it mixing \rm if
there are two adjacent paradoxical balls
 $B_0$, $B_1$ on the same side $\beta_\iota$ of $\beta$, such that
the subpath between the first and the last visit to $B_0\cup B_1$,
intesects at least one paradoxical ball on the other side $\beta_{1-\iota}$ of $\beta$.
We also say in this case that $\psi$ \it mixes \rm the pair $\{B_0,B_1\}$.

On Figure 3 a piece of a mixing path is shown in blue.
\vskip3pt
Denote the set $Q\setminus(Q_{\fam0missing}\cup Q_{\fam0mixing})$ by $\Cal A$.
\subsection{The estimate for $\mathbb P_{\partial\beta}(\Cal A)$}
Let $B_p\leftrightharpoons B_{0,p}\cup B_{1,p}$, $B_q\leftrightharpoons B_{0,q}\cup B_{1,q}$. Denote\vskip3pt
 $\Sigma\leftrightharpoons\{\sigma\in\langle\mathcal S\rangle:$ the path $(x,\sigma)$ ends in
its first visit to $B_p\cup B_q\}$.\vskip3pt
For $(u,\iota)\in\{p,q\}\times\{0,1\}$ let
$\Sigma_{u,\iota}\leftrightharpoons\{\sigma\in\Sigma:x{\cdot}\boldsymbol\pi_\sigma\in B_{u,\iota}\}$,\newline
$\Theta_{u,\iota}\leftrightharpoons\{\theta\in\langle\mathcal S\rangle:$ the path $(z,\theta^{-1})$ ends
in its first visit to $B_{u,\iota}\}$. Denote
\vskip3pt
$\Omega\leftrightharpoons\Sigma_{p,0}\times\Theta_{p,1}\cup\Sigma_{p,1}\times\Theta_{p,0}\cup
\Sigma_{q,0}\times\Theta_{q,1}\cup\Sigma_{q,1}\times\Theta_{q,0}$.\vskip3pt
For $\omega=(\sigma,\theta)$ denote\newline
$(10)$\quad
$g_\omega\leftrightharpoons\boldsymbol\pi(\sigma^{-1}){\cdot}x^{-1}z{\cdot}\boldsymbol\pi(\theta^{-1})$,\hfil\penalty-10000
$\mathcal H_\omega\leftrightharpoons\sigma{\cdot}\boldsymbol\pi^{-1}(g_\omega){\cdot}\theta$.\newline
The sets $\Cal H_\omega$ ($\omega\in\Omega$) are pairwise disjoint by definition of $\Sigma$.

We will find an upper bound for $\mathbb P_{\partial\beta}\Cal A$ by estimation of the
conditional probabilities $\mathbb P_{\partial\beta}(\Cal A|\Cal H_\omega)$ using
Lemma 2.9 and the total probability theorem.

Since each path $\psi\in\Cal A$ is not missing, it belongs to a unique set $\Cal H_\omega$.
So we have a map $\Cal A\overset\lambda\to\Omega$ such that $\psi{\in}\Cal H_\omega$ where $\omega={\lambda(\psi)}=(\sigma,\theta)$
for each $\psi\in\Cal A$.
In this case $\psi$ has the form $\sigma\varphi\theta$
where $\boldsymbol\pi(\varphi){=}g_\omega$.
This $\varphi$ is uniquely defined and we call it the \textit{middle part} of $\psi$.

For $\iota\in\{0,1\}$ denote $B_\iota\leftrightharpoons{\cup}\{B_{\iota,t}:t\in T\}$.

For a path $\psi\in Q$ denote by $N_\iota(\psi)$ the set 
$\psi^{-1}B_\iota,$ $\iota\in\{0,1\}$ of the moments of visits to the paradoxical balls
on the side $\beta_\iota$.
We write `$N_\iota(\psi)<N_{1-\iota}(\psi)$' if
every element of $N_\iota(\psi)$ is less than every element of $N_{1-\iota}(\psi)$.
\vskip3pt
\textsc{Lemma.} \textsl{If $\psi\in\Cal A$ then $N_\iota(\psi)<N_{1-\iota}(\psi)$ for some
$\iota\in\{0,1\}$.}
\vskip3pt
\textit{Proof.}
Suppose that this is not is not true.
Then there exist $\iota\in\{0,1\}$, $s\in N_\iota(\psi)$,
$s_-,s_+\in N_{1-\iota}(\psi)$ such that $s_-<s<s_+$.
Hence, both sets\hfil\penalty-10000
 $T_-\leftrightharpoons\{t\in T:\exists s_-<s:\psi_{s_-}\in B_{1-\iota,t}\}$,
$T_+\leftrightharpoons\{t\in T:\exists s_+>s:\psi_{s_+}\in B_{1-\iota,t}\}$ are nonempty.\hfil\penalty-10000
Since $\psi$ is not missing, $T_-\cup T_+=T$.
It follows that there exist $t_-\in T_-$, $t_+\in T_+$ such that $|t_--t_+|=d$.
So $\psi$ is mixing contradicting the hypothesis.\hfill$\square$
\vskip3pt
On Figure 4 a piece of a typical path $\psi\in\Cal A$ is shown in blue. For this path we have
$N_1(\psi)<N_0(\psi)$.

\noindent
\begin{picture}(0,205)(-40,-46)
\put(-14,-35){
\pdfximage{\mpPath 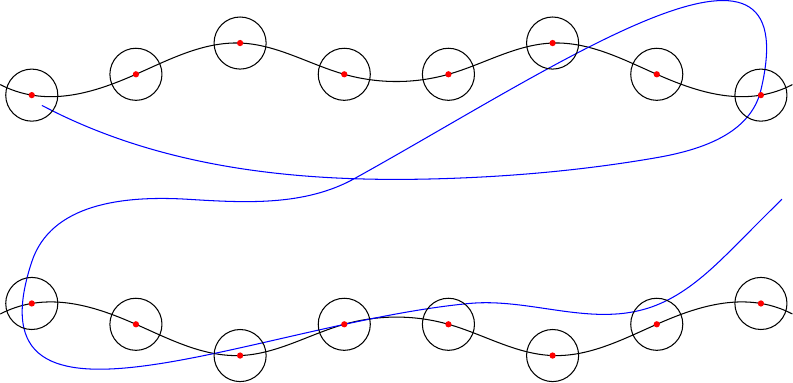}\pdfrefximage\pdflastximage}
\put(5,20){\vector(0,1){68}}
\put(5,84){\vector(0,-1){66}}
\put(356,18){\vector(0,1){71}}
\put(356,88){\vector(0,-1){70}}
\put(27,60){\vector(1,0){325}}
\put(270,60){\vector(-1,0){262}}
\put(175,58){\makebox(0,0)[ct]{$L$}}
\put(175,-35){\makebox(0,0)[cc]{Figure 4}}
\put(0,60){\makebox(0,0)[rc]{$M$}}
\put(12,122){\makebox(0,0)[rc]{$B_{1,p}$}}
\put(8,-15){\makebox(0,0)[rc]{$B_{0,p}$}}
\put(366,122){\makebox(0,0)[rc]{$B_{1,q}$}}
\put(366,-17){\makebox(0,0)[rc]{$B_{0,q}$}}
\put(358,60){\makebox(0,0)[lc]{$M$}}
\end{picture}
\vskip3pt
Let $\psi\in\Cal A$, $\lambda(\psi)=(\sigma,\theta)$, $\psi=\sigma{\cdot}\varphi{\cdot}\theta$,
$x'\leftrightharpoons x{\cdot}\boldsymbol\pi_\sigma\in B_{u,\iota}$
and $z'\leftrightharpoons z{\cdot}\boldsymbol\pi_{\theta^{-1}}\in B_{u,1-\iota}$.\hfil\penalty-10000
Thus $|x'\hbox{-}z'|\leqslant M+2r\leftrightharpoons c\overset{(7)}\geqslant 1$.\hfil\penalty-10000
By Lemma, the subpath of $\psi$
between the first visit to $x'$ and the first visit to $z'$ intersects
the ball $B_{\iota,v}$ where $\{u,v\}=\{p,q\}$. This subpath is also a subpath of $\varphi$,
hence $\ell_\varphi\geqslant2L-M-4r$, see Figure 4.

For $\omega=(\sigma,\theta)\in\Omega$, we have\hfil\penalty-10000
$(11)$\quad $\mathbb P_{x,z}(\Cal A|\Cal H_\omega)=
\Frac{\mathbb P_{x,z}\{\psi\in\Cal A:\lambda(\psi)=\omega\}}{\mathbb P_{x,z}\Cal H_\omega}\overset{2.9}\leqslant
\mathbb P_{\boldsymbol\iota,g_\omega}\{\varphi\in\boldsymbol\pi^{-1}g_\omega:
\sigma\varphi\theta\in\Cal A,\lambda(\sigma\varphi\theta)=\omega\}\leqslant$\hfil\penalty-10000
$\leqslant\mathbb P_{\boldsymbol\iota,g_\omega}\{\varphi\in\boldsymbol\pi^{-1}g_\omega,\ell_\varphi\geqslant2L-M-4r\}$.

By postulating one more inequality:\newline
$(12)$\quad$M\geqslant\s{max}\left\{1,\Frac{2r(A+2)}{2k-A-1}\right\}$ (the denominator is positive by $(4)$), we obtain\newline
$(13)$\quad$M(2k-A-1)\geqslant2r(A+2)$ and\hfil\penalty-10000
$2L-M-4r\overset{(5)}\geqslant2kM-M-4r=$\hfil\penalty-10000
$=(2k-A-1)M+(A+1)M-M-4r=$\newline
$=AM-4r+(2k-A-1)M\overset{(13)}\geqslant AM-4r+2r(A+2)=
A(M+2r)=Ac$.\hfil\penalty-10000
Since $|g_\omega|\leqslant c$, by condition $\s A$, we have\hfil\penalty-10000
$\mathbb P_{x,z}(\Cal A|\Cal H_\omega)\overset{(11)}\leqslant
\Bbb P_{\boldsymbol\iota,g_\omega}\left\{\varphi\in\boldsymbol\pi^{-1}g_\omega:
\ell_\varphi\geqslant2L-M-4r\right\}\leqslant
\Bbb P_{\boldsymbol\iota,g_\omega}\left\{\varphi\in\boldsymbol\pi^{-1}g_\omega:
\ell_\varphi\geqslant Ac\right\}\leqslant a$.

Since the sets $\Cal H_\omega$ are disjoint by the total probability theorem we have:\newline
$(14)$\quad$\Bbb P_{x,z}\mathcal A=
\displaystyle\sum_\omega\mathbb P_{x,z}(\mathcal A|\mathcal H_\omega){\cdot}\mathbb P_{x,z}\mathcal H_\omega
\leqslant a\displaystyle\sum_\omega\mathbb P_{x,z}\mathcal H_\omega\leqslant a$.
\subsection{The estimate for $\Bbb P_{\partial\beta}Q_{\fam0mixing}$}
Let $\mathcal E\leftrightharpoons\{0,1\}{\times}(T\setminus\{q\})$.\hfil\penalty-10000
For $e=(\iota,t)\in\mathcal E$, let $B_e\leftrightharpoons B_{\iota,t}\cup B_{\iota,t+d}$ the union of
two adjacent paradoxical balls, and\hfil\penalty-10000
$Q_e\leftrightharpoons\{\psi\in\boldsymbol\pi^{-1}(x^{-1}z):(x,\psi)$ separates the pair $\{B_{\iota,t},B_{\iota,t+d}\}$.
We will prove that\hfil\penalty-10000
 $\Bbb P_{x,z}Q_e\leqslant\Frac b{2(AB+B+2)}$ for $M\gg0$.

Denote\newline
$\Sigma\leftrightharpoons\{\sigma\in\langle\mathcal S\rangle:$ the path $(x,\sigma)$ ends at its first
visit to $B_e\}$;\hfil\penalty-10000
$\Theta\leftrightharpoons\{\theta\in\langle\mathcal S\rangle:$ the path $(z,\theta^{-1})$ ends at its
first visit to $B_e\}$.\hfil\penalty-10000
$\Omega\leftrightharpoons\Sigma\times\Theta$.

The words in $\Sigma$ are pairwise incomparable with respect to `$\leqslant_{\fam0l}$'
the words in $\Theta$
are pairwise incomparable with respect to `$\leqslant_{\fam0r}$' (see Subsection 2.6).
Thus, by the Lemma from 2.6, the sets $\sigma\langle\mathcal S\rangle\theta$, ($(\sigma,\theta)\in\Omega$) are pairwise disjoint.

So $Q_e\subset{\cup}\{\sigma\langle\Cal S\rangle\theta:(\sigma,\theta)\in\Omega\}$ and we can
apply the total probability theorem.

By definition of paradoxical balls, for each $\omega\in\Omega$
one has
$|g_\omega|\leqslant d+2r\leftrightharpoons c\geqslant1$, where $g_\omega$ is from $(10)$,\newline
$\Bbb P_{x,z}(Q_e|\Cal H_\omega)=
\Bbb P_{\boldsymbol\iota,g_\omega}\{\varphi\in\boldsymbol\pi^{-1}g_\omega:
x{\cdot}\boldsymbol\pi_\sigma{\cdot}\s{Im}\varphi\cap B_{1-\iota}\ne\varnothing\}\leqslant$\newline
$\leqslant
\Bbb P_{\boldsymbol\iota,g_\omega}\{\varphi\in\boldsymbol\pi^{-1}g_\omega:
\ell_\varphi\overset{\fam0Lemma\ 2.3}\geqslant M-4r\}$.

Postulating\newline
$(14)$\quad$M\geqslant2r(B+2)(2kB+1)$, we obtain\newline
$(16)$\quad$\Frac{2r(B+2)}M\overset{(15)}\leqslant\Frac1{2kB+1}$;\hfil\penalty-10000
$(17)$\quad$1-\Frac{2r(B+2)}M\overset{(16)}\geqslant1-\Frac1{2kB+1}=\Frac{2kB}{2kB+1}$;\hfil\penalty-10000
$(18)$\quad$\Frac1{1-\Frac{2r(B+2)}M}\overset{(17)}\leqslant\Frac{2kB+1}{2kB}$;\hfil\penalty-10000
$(19)$\quad$\Frac1{1-\Frac{2r(B+2)}M}-1\overset{(18)}\leqslant\Frac1{2kB}$;\hfil\penalty-10000
$(20)$\quad$B'-B\overset{(19)}\leqslant\Frac1{2k}$, where
$B'\leftrightharpoons\Frac B{1-\Frac{2r(B+2)}M}$;\hfil\penalty-10000
$(21)$\quad$\Frac LM{\cdot}B'\overset{(5)}\leqslant2kB'\leqslant2kB+1\overset{(4)}<n-1$;\hfil\penalty-10000
$(22)$\quad$M\overset{(8),(21)}>dB'\overset{(20)}=\Frac{dB}{1-\Frac{2r(B+2)}M}=\Frac{dBM}{M-2r(B+2)}$;\hfil\penalty-10000
$(23)$\quad$dB\overset{(22)}<M-2r(B+2)$;\hfil\penalty-10000
$B(d+2r)\overset{(23)}<M-4r$;\hfil\penalty-10000
$\ell_\varphi\geqslant M-4r>Bc$ where $c=d+2r\overset{(8)}\geqslant1$.

Thus, by Condition $\s B$, $\Bbb P_{x,z}(Q_e|\Cal H_{g_\omega,\sigma,\theta})<\Frac b{2(AB+B+2)}$.
By the total probability theorem, we have\newline
$\Bbb P_{x,z}Q_e=\displaystyle\sum_\omega\mathbb P_{x,z}(Q_e|\mathcal H_{g_\omega,\sigma,\theta}){\cdot}\mathbb P_{x,z}\mathcal H_{g_\omega,\sigma,\theta}
<\Frac b{2(AB+B+2)}$.

 The parameter $e$ has its values in the
set $\mathcal E$ of cardinality $2(n-1)$. So\hfil\penalty-10000
$(24)$\quad$\Bbb P_{x,z}Q_{\fam0mixing}\leqslant\Frac{2(n-1)b}{2(AB+B+2)}\overset{(2)}<b$.

By $(9)$, $(14)$, $(24)$ we have
$1=\Bbb P_{x,z}Q\leqslant\Bbb P_{x,z}Q_{\fam0missing}+
\Bbb P_{x,z}Q_{\fam0mixing}+\Bbb P_{x,z}(Q\setminus(Q_{\fam0missing}{\cup}Q_{\fam0mixing}))<1$
which is a contradiction.

Proposition 6 and Theorem 3 are proved.


\begin{thebibliography}{AAAAA}
\bibitem[Anc87] {Anc87} 
A. Ancona\sl. Positive harmonic functions and hyperbolicity\rm. In: Potential Theory
--- Surveys and Problems (Prague, 1987). Lecture Notes in Math. 1344, p. 123. Springer Berlin.
\bibitem[Anc88] {Anc88} 
A. Ancona\sl. Th\'eorie du potentiel sur les graphes et les vari\'et\'es\rm, in \'Ecole d' \'et\'e de
Probabilit\'es de Saint-Flour XVIII 1988, Lecture Notes in Math. 1427, Springer,
1990, 1--112.

\bibitem [Fu71]{Fu71} H. Furstenberg . Random walks and discrete subgroups of Lie groups\rm. In: Advances in
Probability and Related Topics 1 (ed. by P. Ney), 1-63. Dekker, New York 1971.
\bibitem[GGPY21]{GGPY21} I. Gekhtman, V. Gerasimov, L. Potyagailo and W. Yang\sl.
Martin boundary covers Floyd boundary\rm. Invent. Math. \bf 223 \rm(2021) 759--809.
\bibitem[GP24]{GP24}
V. Gerasimov and L. Potyagailo\sl. Integral criteria for hyperbolicity of graphs and groups\rm.
S\~ao Paulo Journal of Math. Sci. \textbf{18}, no 2 (Memorial Volume for Sasha Anan'in),
676--695.
\bibitem[Gr87]{Gr87}
M. Gromov\sl, Hyperbolic groups\rm, in:
``Essays in Group Theory'' (ed. S.~M.~Gersten) M.S.R.I. Publications No.~8,
Springe-Verlag (1987) 75--263.
\bibitem [HM14]{HM14}
P. Haissinsky and P. Matthieu\sl, La conjecture de Baum-Connes pour les groupes hyperboliques par les marches aléatoires\rm, preprint, 2014.
\bibitem[Kes59]{Kes59}
H. Kesten\sl. Symmetric random walks on groups\rm. Trans. Amer. Math. Soc. \textbf{92} (1959), 336--354.
\bibitem[La02]{La02}
 V. Lafforgue\sl. K-th\'eorie bivariante pour les alg\'ebres de Banach et conjecture de Baum-Connes\rm. Invent. Math. 149 (2002), 1--95.
\bibitem[MY02]{MY02}
 I. Mineyev \& G. Yu\sl. The Baum-Connes conjecture for hyperbolic groups\rm. Invent. Math. 149 (2002), 97--122.
\bibitem[Pap95]{Pap95}
P. Papasoglu\sl. Strongly automatic groups are hyperbolic\rm. Invent. Math. 121, 323--334 (1995).
\bibitem[Va86]{Va86}
N. Varopoulos\sl. Th\'eorie du potentiel sur des groupes et des vari\'et\'es\rm, C. R.
Acad. Sci. Paris, Serie I 302 (1986), 203--205.
\bibitem[Va97]{Va97} A. Valette. An Introduction to the Baum-Connes Conjecture.
With an Appendix by G. Mislin. From notes taken by I. Chatterji.
Lectures Notes in Mathematics, ETH Z\"urich, Birkh\"auser
\bibitem[Wo00]{Wo00}
W. Woess\sl, Random Walks on Infinite Graphs and Groups\rm. Cambridge University Press, 2000.
\end{thebibliography}
\end{document}